\documentclass[a4paper,12pt,reqno]{amsart}
\usepackage{geometry}
\usepackage[utf8]{inputenc}
\usepackage{amsmath, mathtools, comment, dsfont, amssymb, enumitem}
\usepackage[natbib=true, style=alphabetic, sorting=nyvt, backend=biber]{biblatex}
\DeclareFieldFormat{extraalpha}{#1}
\DeclareLabelalphaTemplate{
  \labelelement{
    \field[final]{shorthand}
    \field{label}
    \field[strwidth=3,strside=left,ifnames=1]{labelname}
    \field[strwidth=1,strside=left]{labelname}
  }
}
\usepackage[colorlinks, allcolors=blue]{hyperref}
\newtheorem{theorem}{Theorem}
\newtheorem*{xremark}{Remark}
\newtheorem{remark}{Remark}
\newtheorem{proposition}{Proposition}
\newtheorem{lemma}{Lemma}

\newtheorem{nc}{Normal Comparison Result}
\newtheorem{na}{Normal Approximation Result}
\newtheorem{es}{Exponential Sum Result}
\newcommand{\N}{\mathbb{N}}
\newcommand{\R}{\mathbb{R}}
\newcommand{\E}{\mathbb{E}}
\newcommand{\p}{\mathbb{P}}
\newcommand{\Z}{\mathbb{Z}}
\newcommand{\C}{\mathbb{C}}
\newcommand{\e}{\mathrm{e}}

\newcommand{\Mod}[1]{\ \mathrm{mod}\ #1}
\newcommand*\conj[1]{\overline{#1}}
\begin{document}
\author{Seth Hardy}
\title[Exponential sums with random multiplicative coefficients]{Bounds for exponential sums with random multiplicative coefficients}
\address{Mathematics Institute, Zeeman Building, University of Warwick, Coventry CV4 7AL, England}
\email{Seth.Hardy@warwick.ac.uk}
\date{\today}
\thanks{The author is supported by the Swinnerton-Dyer scholarship at the Warwick Mathematics Institute Centre for Doctoral Training.}
\begin{abstract}
For $f$ a Rademacher or Steinhaus random multiplicative function, we prove that
\[ \max_{\theta \in [0,1]} \frac{1}{\sqrt{N}} \Bigl| \sum_{n \leq N} f(n) \e (n \theta) \Bigr| \gg \sqrt{\log N} ,\]
asymptotically almost surely as $N \rightarrow \infty$. Furthermore, for $f$ a Steinhaus random multiplicative function, and any $\varepsilon > 0$, we prove the partial upper bound result
\[ \max_{\theta \in [0,1]} \frac{1}{\sqrt{N}} \Bigl| \sum_{\substack{n \leq N \\ P(n) \geq N^{0.8}}} f(n) \e (n \theta) \Bigr| \ll {(\log N)}^{7/4 + \varepsilon},\]
asymptotically almost surely as $N \rightarrow \infty$, where $P(n)$ denotes the largest prime factor of $n$.
\end{abstract}
\maketitle
\section{Introduction}
\subsection{Background}
Let $\mathbf{X} = {(X_n)}_{n \in \N}$ be a sequence of independent random variables with zero mean and unit variance. For many classical choices of $\mathbf{X}$, the limiting behaviour of the random trigonometric polynomials \begin{equation*}Q_{\mathbf{X},N} (\theta) = \frac{1}{\sqrt{N}} \sum_{n \leq N} X_n \e (n \theta), \, \e(x) \coloneqq e^{2\pi i x}, \end{equation*} has been studied: see~\cite{kahane} for examples. A revolutionary paper in this area came from~\citet{SZ}, who studied the case where $\mathbf{X} = {(X_n)}_{n \in \N}$ is a sequence of independent Rademacher random variables; that is, $X_n$ take values $\pm 1$ with equal probability. One of their main results~\cite[Theorem~3.5.2]{SZ} showed that, in this case, $Q_{\mathbf{X},N} (\theta)$ obeys a central limit theorem, establishing that for almost all Rademacher sequences, for $\Theta$ a random variable uniformly distributed on $[0,1]$, $Q_{\mathbf{X},N} (\Theta)$ converges in distribution to standard complex normal random variable. They also showed that $\max_{\theta \in [0,1]} \frac{1}{\sqrt{N}} \bigl| \sum_{n \leq N} X_n \e (n \theta) \bigr| \asymp \sqrt{\log N}$, almost surely, when $N$ is sufficiently large. \\

Recently,~\citet{BNR} investigated whether the behaviour observed by Salem and Zygmund persists when we remove independence between the coefficients $\mathbf{X}={(X_n)}_{n \in \N}$, instead endowing them with a multiplicative dependence structure. More precisely, they studied the sum
\[ P_N (\theta) = \frac{1}{\sqrt{N}} \sum_{n \leq N} f(n) \e (n\theta), \]
where $f : \N \rightarrow \C$ is a Rademacher or a Steinhaus random multiplicative function, defined shortly. They found that $P_N (\theta)$ satisfies a central limit theorem, specifically showing that for almost all Rademacher or Steinhaus random multiplicative functions $f$, for $\Theta$ a uniform random variable on $[0,1]$, $P_N (\Theta)$ converges in distribution to a standard complex normal random variable as $N \rightarrow \infty$~\cite[Theorem~1.2]{BNR}. \\

A Rademacher random multiplicative function $f$ is defined as follows: let ${( f(p) )}_{p \text{ prime}}$ be a sequence of independent Rademacher random variables, set $f(p^\alpha) = 0$ for any $\alpha \geq 2$, and extend this function by multiplicativity so that \[ f(n) = \prod_{p \mid n} f \bigl(p^{v_p (n)} \bigr),\] where $v_p (n)$ is the $p \,$-adic valuation of $n$. Similarly, we define a Steinhaus random multiplicative function $f$ by taking ${(f(p))}_{p \text{ prime}}$ to be independent Steinhaus random variables; that is, random variables uniformly distributed on the unit circle $\{ z \in \mathbb{C} : \, |z| = 1 \}$, and setting \[ f(n) = \prod_{p \mid n} {f(p)}^{v_p(n)}. \] For a detailed overview of the literature on random multiplicative functions see the introduction to~\cite{harperlm} and the references therein. \\

To establish the aforementioned central limit theorem for $P_N (\theta)$,~\citet{BNR} studied the moments $\int_{0}^{1} \E|P_N (\theta)|^{2k} d \theta$, reducing the problem to counting solutions to the system of equations:
\begin{align*}
m_1 \ldots m_k &= m_{k+1} \ldots m_{2k} \\
m_1 + \cdots + m_k &= m_{k+1} + \cdots + m_{2k} ,\end{align*}
for $m_i \leq N$. They found that, when $k$ is of reasonable size relative to $N$, the dominant contribution to this count comes from the diagonal terms where $(m_1,\ldots,m_k)$ is just a permutation of $(m_{k+1},\ldots,m_{2k})$. This contribution agrees with Gaussian moments, giving convergence in distribution by the method of moments\footnote{They also proved that the contribution from off-diagonal terms begins to dominate when $k$ is sufficiently large, showing that the tails are not at all Gaussian.}. Furthermore, they used their moment estimates to investigate \emph{asymptotic almost sure} bounds for $\max_{\theta \in [0,1]} | P_N (\theta)|$, meaning that they obtained upper and lower bounds that hold with probability $1-o(1)$ as $N \rightarrow \infty$. Specifically, they showed (\cite[Theorem~1.3]{BNR}) that
\[ {\biggl(\frac{\log N}{\log_2 N} \biggr)}^{1/6} \leq \max_{\theta \in [0,1]} |P_N (\theta) | \leq \exp\bigl(3 \sqrt{ \log N \log_2 N} \bigr), \]
asymptotically almost surely, where $\log_2 N \coloneqq \log \log N$. In this paper we improve on their lower bound and show the following:
\begin{theorem}\label{T:lb} Let $f$ be a Rademacher or Steinhaus random multiplicative function. We have
\[ \sqrt{\log N} \ll \max_{\theta \in [0,1]} |P_N (\theta)|, \]
with probability $1 - o(1)$ as $N \rightarrow \infty$. In the Rademacher and Steinhaus cases, we calculate that the implicit constants can be taken to be $\frac{4 \sqrt{6}}{29 \pi}$ and $\frac{4}{29}$ respectively.
\end{theorem}
This lower bound is of the conjectured order of magnitude, which we explain in the proceeding section. We also prove the following upper bound result in the Steinhaus case:
\begin{theorem}\label{T:ub} Let $f$ be a Steinhaus random multiplicative function. For any $\varepsilon > 0$, we have
\[ \max_{\theta \in [0,1]} \frac{1}{\sqrt{N}} \Bigl| \sum_{\substack{n \leq N \\ P(n) \geq N^{0.8}}} f(n) \e (n \theta) \Bigr| \ll {( \log N )}^{7/4 + \varepsilon}, \]
with probability $1-o_{\varepsilon}(1)$ as $N \rightarrow \infty$, where $P(n)$ denotes the largest prime factor of $n$.
\end{theorem}   
In this theorem, $o_\varepsilon(1)$ represents a function \emph{depending on $\varepsilon$} that converges to zero as $N \rightarrow \infty$. Since the set $\{ n \leq N \, : P(n) \geq N^{0.8} \}$ has positive asymptotic density, this result handles a positive proportion of the full sum. In our proof of Theorem~\ref{T:ub}, we perform a dyadic decomposition, take a union bound, and use classical bounds for exponential sums. Inefficiencies in performing these steps give a larger bound than the desired size of $\sqrt{\log N}$.
\vspace{8pt}
\begin{xremark}
\normalfont~One could approach the Rademacher case using similar techniques, but this involves technical complications that we point out in Remark~\ref{r:rademachercase}. In the interest of exposition, we do not attempt to resolve these.
\end{xremark}
Our methods for both the upper and lower bounds differ from~\cite{BNR}. In our lower bound, we work directly with a Gaussian approximation, and in the upper bound, we exploit uniformity in the conditional variances (when working over numbers with a large prime factor) to show that the maximum of these is small asymptotically almost surely.
\subsection{Heuristic discussion}\label{s:background}
As mentioned, for independent Rademacher random variables ${(X_n)}_{n \in \N}$,~\citet{SZ} showed that we almost surely have
\[ \max_{\theta \in [0,1]} \frac{1}{\sqrt{N}} \Bigl| \sum_{n \leq N} X_n \e (n \theta) \Bigr| \asymp \sqrt{\log N}. \]
If we fix $\theta$, $ \frac{1}{\sqrt{N}} \bigl| \sum_{n \leq N} X_n \e (n \theta) \bigr|$ is a sum of independent random variables that converge to a Gaussian (or the absolute value of a Gaussian, to be precise). Furthermore, the sums decorrelate when $\theta$ varies by more than $\asymp \frac{1}{N}$, so we should expect the maximum to behave like a maximum of $\asymp N$ independent Gaussians with mean $0$ and unit variance. The asymptotic almost sure order of magnitude $\sqrt{\log N}$ is in accordance with this. 
Unfortunately, this heuristic is far weaker in the random multiplicative case. For `most' $\theta$, one may expect that the increments $f(n) \e (n \theta)$ are roughly independent for each $n$, since the additive function $\e (n \theta)$ should interfere with the multiplicative structure of $f(n)$ when $\theta$ is badly approximable by a low denominator rational. In this case, $\frac{1}{\sqrt{N}} \bigl| \sum_{n \leq N} f(n) \e (n \theta) \bigr| $ should behave like a normal random variable as $f$ varies over all Rademacher or Steinhaus random multiplicative functions. Such a result has been proven (for Steinhaus $f$) in a recent preprint of \citet[Theorem~1.6]{SoXu}, with the essential requirement that $\theta$ satisfies certain Diophantine properties (and only in the `bulk' of the distribution rather than the tails of the distribution). For us, the cases where $\theta$ is close to a low denominator rational provide the main difficulty. As an extreme example, when $\theta = 0$, our sum is $\frac{1}{\sqrt{N}} \bigl| \sum_{n \leq N} f(n) \bigr|$, and we do not have convergence to a non-degenerate Gaussian (this was shown by \citet[Corollary~1]{harperlm} using probabilistic ideas from the theory of multiplicative chaos). In fact, the distribution has considerably longer tails than those of a Gaussian~\cite[Corollary~2]{harperlm}, making large values far more difficult to control. This doesn't cause any issues for the lower bound, since we can restrict our maximum to a smaller set of well-behaved $\theta$ values, but makes understanding the upper bound considerably more challenging.
\subsection{Outline of the proof of the lower bound, Theorem~\ref{T:lb}}\label{sb:lboutline}
For the lower bound, we utilise ideas of \citet{harperalmostsurelf}, which have also found use in a recent work of~\citet{chowla}. The first step is to restrict the maximum to $\theta$ in some nice set, $\mathcal{A}$, and consider only the real part of our sum. We then show that there asymptotically almost surely exists some large proportion of $\theta \in \mathcal{A}$ such that $\frac{1}{\sqrt{N}} \sum_{n \leq N, P(n) \leq N^{6/7}} f(n) \e (n \theta)$ is negligible as $N \rightarrow \infty$. That is, there asymptotically almost surely exists some large set $\tilde{\mathcal{A}} \subseteq \mathcal{A}$ such that
\[ \max_{\theta \in [0,1]} |P_N (\theta)| \geq \max_{\theta \in \tilde{\mathcal{A}}} \frac{1}{\sqrt{N}} \Re \Bigl( \sum_{N^{6/7} < p \leq N} f(p) \sum_{\substack{m \leq N/p \\ P(m) \leq N^{6/7}}} f(m) \e (mp \theta) \Bigr) , \]
up to some negligible term. One should note that $\tilde{\mathcal{A}}$ depends only on the random variables ${\bigl( f(p) \bigr)}_{p \leq N^{6/7}}$, so conditioning on ${\bigl( f(p) \bigr)}_{p \leq N^{6/7}}$ fixes the set  $\tilde{\mathcal{A}}$. Having conditioned, the term in the parenthesis is just a sum of independent random variables weighted by coefficients $\sum_{m \leq N/p} f(m) \e (mp\theta)$, and, in the Steinhaus case, for example, we can approximate the above by a maximum of a multivariate Gaussian distribution with mean $0$, variances \[ \frac{1}{2N} \sum_{N^{6/7} < p \leq N} \Bigl| \sum_{m \leq N/p} f(m) \e (mp \theta) \Bigr|^2 , \] and covariances \[ \frac{1}{2N} \Re \Bigl( \sum_{N^{6/7} < p \leq N} \sum_{m_1 \leq N/p} \sum_{m_2 \leq N/p} f(m_1) \conj{f(m_2)} \e \bigl(p(m_1 \theta_1 - m_2 \theta_2)\bigr) \Bigr). \] This step is performed using Normal Approximation Result~\ref{na} (found in~\cite{harperalmostsurelf} under the same name). The components of this multivariate Gaussian, which correspond to different $\theta$ values, will behave similarly to independent and identically distributed normal random variables, as long as the following hold:
\begin{enumerate}[label = (\roman*)]
    \item The variances are approximately equal with high probability.
    \item The coefficients $\sum_{m \leq N/p} f(m) \e (mp \theta)$ are (with high probability)  weakly correlated for different values of theta.
\end{enumerate}
To obtain a good bound, $\tilde{\mathcal{A}}$ should contain at least $N^{\beta}$ elements for some $\beta>0$, since a maximum of $N^{\beta}$ Gaussians will be $\geq C_\beta \sqrt{\log N}$. These conditions can be satisfied if we make an initial choice of $\mathcal{A}$ that allows for effective use of the following result:
\begin{es}[Davenport~\cite{davenport}, Chapter 25]\label{es}
For $\beta \in \mathbb{R}$, if there exist $a,q \in \mathbb{N}$ such that $(a,q)=1$, and $\Bigl| \beta - \frac{a}{q} \Bigr| \leq \frac{1}{q^2}$, then
\[ \sum_{n \leq N} \Lambda(n) \e (n \beta) \ll \bigl(N q^{-1/2} + N^{4/5} + N^{1/2} q^{1/2}\bigr) {(\log N)}^{4}. \]
\end{es}
Once the above conditions are established, one can apply results on the maxima of multivariate Gaussians with small correlations (specifically Normal Comparison Result~\ref{nc}, again found in~\cite{harperalmostsurelf} under the same name) to capture the Gaussian behaviour, proving the lower bound. Throughout the proof, we will use the following lemma to evaluate any moments of sums involving random multiplicative functions.
\begin{lemma}\label{L:Expectation}
Let $f$ be a Rademacher or Steinhaus random multiplicative function, for any real $k \geq 1$, $a_n \in \C$, we have \[ \E \Bigl| \sum_{n \leq N} a_n f(n) \Bigr|^{2k} \leq {\biggl( \sum_{n \leq N} \tau_{2\lceil k \rceil-1} (n) |a_n|^2 \biggr)}^k , \]
where $\tau_k$ denotes the $k$-divisor function, $\tau_k (n) = \# \{(b_1,\ldots,b_k): b_1 b_2 \ldots b_k = n, \, b_i \in \N \}$.
\end{lemma}
\begin{proof}
This is~\cite[Probability Result 2.3]{harperhm}.
\end{proof}
\vspace*{3pt}
\begin{xremark}
\normalfont~Little attempt has been made to optimise the constants in the statement of Theorem~\ref{T:lb}. We remark that alterations to the smoothness parameter $6/7$ and the set $\mathcal{A}$ can certainly lead to improvements.
\end{xremark} 
\subsection{Outline of the proof of the upper bound, Theorem~\ref{T:ub}}\label{sb:uboutline}
In contrast to the lower bound, we cannot simply restrict to a small set of ``nice'' $\theta$ values, seeing as any $\theta$ value could maximise $|P_N (\theta)|$. However, as noted in~\cite{BNR}, we can restrict to a set of discrete points by an application of Bernstein's Inequality (see exercise 7.16 of~\cite[Section~1]{HAkatznelson}), which tells us that, for $p_N (\theta) = \sum_{n \leq N} a_n \e(n \theta)$ with ${(a_n)}_{n \leq N}$ complex coefficients, we have
\[ \max_{\theta \in [0,1]} |p'_N(\theta)| \leq 2 \pi N \max_{\theta \in [0,1]} |p_N (\theta)|. \]
It follows (see for example~\cite[Lemma~4.2]{BNR}) that
\[ \max_{\theta \in [0,1]} |p_N (\theta)| \leq 2 \max_{\theta \in \mathcal{D}} |p_N (\theta)| ,\]
for any set $\mathcal{D}$ with the property that, if $\theta \in [0,1]$, then there exists $\theta' \in \mathcal{D}$ such that $\|\theta - \theta'\| \leq \frac{1}{4 \pi N}$ (where here and throughout $\| . \|$ denotes the distance to the nearest integer). This restriction will later allow us to apply the union bound. We take the discretised set of points $\mathcal{D}$ to be of the form
\[ \mathcal{D} = \biggl\{ \frac{a}{q} + \frac{j}{4 \pi N} : q \leq Q, \, 1 \leq a \leq q \, \, \mathrm{ with } \, \, (a,q)=1, \, j \in \biggl[\frac{-4 \pi N}{q Q}, \frac{4 \pi N}{q Q}\biggr] \cap \mathbb{Z} \biggr\} , \]
for some function $Q$ depending on $N$. Note that this set satisfies the $\frac{1}{4 \pi N}$-spacing property by Dirichlet's approximation theorem. Therefore, to prove Theorem~\ref{T:ub}, it suffices to prove the corresponding statement about $\frac{1}{\sqrt{N}} \bigl| \sum_{\substack{n \leq N \\ P(n) \geq N^{0.8}}} f(n) \e(n \theta) \bigr|$, where the maximum is taken over $\mathcal{D}$ instead. We note that
 \[ \sum_{\substack{n \leq N \\ P(n) \geq N^{0.8}}} f(n) \e (n \theta) = \sum_{N^{0.8} \leq p \leq N} f(p) \sum_{m \leq N/p} f(m) \e(m p \theta), \]
and the innermost sum on the right hand side is of fixed length $r$ for $N/(r+1) < p \leq N/r$. Therefore, we can write our sum as
\[ \frac{1}{\sqrt{N}} \sum_{r \leq N^{0.2}} \sum_{N/(r+1) < p \leq N/r} f(p) \sum_{m \leq r} f(m) \e (m p \theta).\]
The proof heavily relies on the fact that we can rewrite our sum in this way: we will exploit the averaging of the innermost sum as the prime $p$ varies. Conditioning on ${\bigl( f(p) \bigr)}_{p \leq N^{0.2}}$, the inner sum becomes fixed and this is just a sum of (weighted) independent random variables. Letting $\tilde{\E}$ denote the expectation conditioned on ${\bigl( f(p) \bigr)}_{p \leq N^{0.2}}$ and taking conditional moments, we find that
\[ \tilde{\E} \Bigl| \frac{1}{\sqrt{N}} \sum_{\substack{n \leq N \\ P(n) \geq N^{0.8}}} f(n) \e (n \theta) \Bigr|^{2k} \leq {\Bigl( \frac{k}{N} \sum_{r \leq N^{0.2}} \sum_{N/(r+1) < p \leq N/r} \bigl| \sum_{m \leq r} f(m) \e (m p \theta) \bigr|^2 \Bigr)}^k , \]
and by an application of Markov's inequality (see section~\ref{s:completionoft2}), it suffices to upper bound the maximum of the conditional variances
\[ \max_{\theta \in \mathcal{D}} \frac{1}{N} \sum_{r \leq N^{0.2}} \sum_{N/(r+1) < p \leq N/r} \bigl| \sum_{m \leq r} f(m) \e (m p \theta) \bigr|^2, \]
asymptotically almost surely as $N \rightarrow \infty$. To do this, we proceed by analysing the distribution of $\bigl( p \theta \bigr)$ as $p$ varies over certain subintervals contained in $\bigl( \frac{N}{r+1}, \frac{N}{r}\bigr]$. With $\theta = \frac{a}{q} + \frac{j}{4 \pi N}$, we have $p \theta = \frac{ap}{q} + \frac{pj}{4 \pi N}$: now if $p$ lies in a small interval, say $[c,d)$ (which will depend on $r,j$ and $N$) so that $\frac{pj}{4 \pi N}$ varies by less than $\asymp  1/N^{2/5}$, say, we have 
\begin{equation} \sum_{p \in [c,d)} \Bigl| \sum_{m \leq r} f(m) \e \Bigl( m \Bigl( \frac{pa}{q} + \frac{pj}{4 \pi N} \Bigr) \Bigr) \Bigr|^2 \approx \sum_{p \in [c,d)} \Bigl| \sum_{m \leq r} f(m) \e \Bigl( m \Bigl( \frac{pa}{q} + \frac{cj}{4 \pi N} \Bigr) \Bigr) \Bigr|^2. \tag{1.1} \label{equ:1.1} \end{equation}
 Therefore, we begin by dividing the intervals $\bigl( \frac{N}{r+1}, \frac{N}{r}\bigr]$ into subintervals of length $\frac{N^{3/5}}{|j|}$. Of course, in the cases when $j=0$, we need not divide these intervals at all.

We now turn our attention to analysing the behaviour of $\frac{pa}{q}$ over these new intervals. The key thing to notice is that, assuming these intervals are sufficiently long, $\frac{pa}{q}$ distributes uniformly over reduced fractions with denominator $q$, regardless of the value of $a$, seeing as the Brun--Titchmarsh theorem tells us that for any $x > 2$, and any $q \leq y$, we have
\[ \pi(x+y; q, a) - \pi(x;q,a) \leq \frac{2y}{\varphi(q) \log (y/q)}. \]
Under restricted conditions, we can therefore upper bound~\eqref{equ:1.1} by
\[ \ll \frac{(d-c)}{\varphi(q) \log \bigl( (d-c)/q \bigr)} \sum_{(b,q)=1} \Bigl| \sum_{m \leq r} f(m) \e \Bigl( m \Bigl( \frac{b}{q} + \frac{cj}{4 \pi N} \Bigr) \Bigr) \Bigr|^2, \]
which no longer depends on the parameter $a$. This uniformity gives a large saving since the maximum over $\mathcal{D}$ can be replaced by a maximum over $\theta$ with $q \leq Q$ and $|j| \leq \frac{4 \pi N}{qQ}$. When applying a union bound, this decreases the number of terms considered by a factor of $Q$ (up to logarithms). To estimate the resulting probabilities, we expand out the square, remove the diagonal terms, and apply Markov's inequality with second moments, utilising the simple but important fact that
\[ \E \Biggl[ \sum_{\substack{m_1, m_2 \leq r, m_1 \neq m_2 \\ m_3, m_4 \leq r, m_3 \neq m_4}} f(m_1) \conj{f(m_2)} \, \conj{f(m_3)} f(m_4) \Biggr] = \sum_{\substack{m_1, m_2 \leq r, m_1 \neq m_2 \\ m_3, m_4 \leq r, m_3 \neq m_4 \\ m_1 m_4 = m_2 m_3}} 1, \]
which follows from the orthogonality of Steinhaus $f(n)$. The proof is completed by showing that the resulting exponential sums are small, which can be performed using standard techniques. In particular, we make use of the following result:
\begin{es}\label{es2}
For any $M \in \N$ and nonzero $\theta \in \R$, 
\[ \sum_{M' \leq n \leq 2M} \e \biggl( \frac{\theta}{n+1} \biggr) \ll  \frac{|\theta|^{1/2}}{M^{1/2}} + \frac{M^{3/2}}{|\theta|^{1/2}} , \]
holds uniformly for $M' \in [M,2M]$.
\end{es}
\begin{proof}
This follows from Theorem 2.2 of~\cite{vdcmethod}.
\end{proof}
We remark that this bound is only nontrivial when $M \leq |\theta| \leq M^3$. This will be of little consequence, since the critical case we need to handle will be when $M$ is large (roughly of size $N$) and $|\theta|$ lies in this range, so considering the trivial bound gives little improvement overall. \\
\begin{remark}~\label{r:interplay}
\normalfont~The limits of this method arise from the fact that the Brun--Titchmarsh bound only holds for $q$ smaller than the length of the interval that the primes run over. In our proof, the splitting of our prime sum means that we will require $\frac{N^{3/5}}{|j|} \geq q$, for all nonzero $j$, and $ \frac{N}{r(r+1)} \geq q $, for the case $j=0$. Seeing as $q |j| \leq \frac{4 \pi N}{Q}$, and $r \leq N^{0.2}$, we deduce that both conditions are satisfied if $4 \pi N^{2/5} \leq Q \leq \frac{N^{3/5}}{2}$, say. In our proof, we take $Q = 4 \pi \sqrt{N}$. Technical changes can be made to slightly reduce the smoothness condition $P(n) \geq N^{0.8}$ in our sum (say, to $P(n) \geq N^{5/7}$). We suggest how such changes could be made in Remark~\ref{r:optimising}. However, our techniques could not be pushed to handle $n$ with $P(n) < N^{0.5}$, seeing as, in such a case, the sum over primes satisfying $\frac{N}{r+1} < p \leq \frac{N}{r}$ can be empty. 
\end{remark}
\section{Lower Bound}\label{S:LB}
\subsection{Approximating by a normal}
As described, the first step in the proof is to split up our sum into two pieces. One piece will be small asymptotically almost surely for a large proportion of $\theta$, and the other will take an approximately Gaussian shape. To illuminate the Gaussian behaviour, we employ the following results of \citet{harperalmostsurelf}, referred to there under the same names:
\begin{na}\label{na}
Suppose that $m \geq 1$, and that $\mathcal{H}$ is a finite non-empty set. Suppose that for each $1 \leq i \leq m$ and $h \in \mathcal{H}$ we are given a deterministic co-efficient $c(i,h) \in \C$. Finally, suppose that ${(V_i)}_{1 \leq i \leq m}$ is a sequence of independent, mean zero, complex random variables, and let $Y={(Y_h)}_{h \in \mathcal{H}}$ be the $\#\mathcal{H}$-dimensional random vector with components $Y_h \coloneqq\Re (\sum_{i=1}^m c(i,h) V_i)$.
If $Z = {(Z_h)}_{h \in \mathcal{H}}$ is a multivariate normal random vector with the same mean vector and covariance matrix as $Y$, then for any $u \in \R$ and any small $\eta > 0$ we have
\begin{multline*} 
\p(\max_{h \in \mathcal{H}} Y_h \leq u ) \leq \p (\max_{h \in \mathcal{H}} Z_h \leq u + \eta ) +  \\
O \Biggl(\frac{1}{\eta^3} \sum_{i=1}^m \E |V_i|^3 {(\sum_{h \in \mathcal{H}} |c(i,h)|)}^3 + \frac{1}{\eta^2} \sum_{g,h \in \mathcal{H}} \sqrt{\sum_{i=1}^m |c(i,g)|^2 |c(i,h)|^2 \E |V_i|^4 }\Biggr).
\end{multline*}
\end{na}
\begin{nc}\label{nc}
Suppose that $n \geq 2$ and that $\varepsilon>0$ is sufficiently small (i.e.\ less than a certain absolute constant). Let $X_1, \ldots, X_n$ be mean zero, variance one, jointly normal random variables, and suppose $\E X_i X_j \leq \varepsilon$ for $i \neq j$.
Then for any $100 \varepsilon \leq \delta \leq 1/100$, we have
\[ \p \biggl( \max_{1 \leq i \leq n} X_i \leq \sqrt{(2 - \delta) \log n} \biggr) \ll e^{- \Theta(n^{\delta/20}/\sqrt{\log n})} + n^{-\delta^2 / 50 \varepsilon} . \]
\end{nc}
\begin{proof}[Proof of Theorem~\ref{T:lb}.]
Let $f$ be either a Steinhaus or a Rademacher random multiplicative function. As mentioned, for the lower bound we can restrict our maximum to any set $\mathcal{A} \subseteq [0,1]$. We choose this set to allow for effective use of Exponential Sum Result~\ref{es} later in the argument. Given any $N$, by Bertrand's postulate we can find some prime $q$ such that $N^{1/2} \leq q \leq 2 N^{1/2}$. We let \begin{align} \mathcal{A} = \Bigl\{ \frac{p_{r+i}}{q} : \, p_r < N^{1/7} \text{ maximal} , \,  1 \leq i \leq N^{1/8} \Bigr\}, \label{equ:2.1} \tag{2.1} \end{align}
so that $p_{r+1}, p_{r+2}, \ldots$ are the first $N^{1/8}$ primes larger than $N^{1/7}$. The first step in our proof is to decompose the sum as
\[ \max_{\theta \in \mathcal{A}} |P_N (\theta)| = \max_{\theta \in \mathcal{A}} \Bigl| \frac{1}{\sqrt{N}} \sum_{N^{6/7} < p \leq N} f(p) \sum_{m \leq N/p} f(m) \e (mp \theta) + \frac{1}{\sqrt{N}} \sum_{\substack{n \leq N \\ P(n) \leq N^{6/7}}} f(n) \e (n \theta) \Bigr|. \]
We then proceed similarly to \citet{harperalmostsurelf}, showing that, with high probability, the second term is small for \emph{most} $\theta \in \mathcal{A}$. That is, 
\begin{align*}
&\p \Bigl( \exists \text{ random } \tilde{\mathcal{A}} \subseteq \mathcal{A} \text{ with } |\tilde{\mathcal{A}}| \geq 0.99 |\mathcal{A}|  : \frac{1}{\sqrt{N}} \Bigl| \sum_{\substack{n \leq N \\ P(n) \leq N^{6/7}}} f(n) \e (n \theta) \Bigr| < {(\log N)}^{1/10} \, \, \forall \theta \in \tilde{\mathcal{A}} \Bigr) \\
& \geq 1 - \frac{100}{|\mathcal{A}|} \E \# \Bigl\{ \theta \in \mathcal{A} \, : \, \frac{1}{\sqrt{N}} \Bigl| \sum_{\substack{n \leq N \\ P(n) \leq N^{6/7}}} f(n) \e (n \theta) \Bigr| \geq {(\log N)}^{1/10} \Bigr\} \\ 
& \geq 1 - \frac{100}{N |\mathcal{A}| {(\log N)}^{1/5}} \sum_{\theta \in \mathcal{A}} \E \Bigl| \sum_{\substack{n \leq N \\ P(n) \leq N^{6/7}}} f(n) \e (n \theta) \Bigr|^2 \geq 1 - \frac{100}{{(\log N)}^{1/5}}.
\end{align*}
For the remainder of the proof, we may restrict ourselves to the above event (which is determined entirely by the random variables ${\bigl(f(p)\bigr)}_{p \leq N^{6/7}}$) since it occurs with probability $1-o(1)$ as $N \rightarrow \infty$. We introduce the notation $\tilde{\p}$ and $\tilde{\E}$ to denote the probability and expectation conditioned on ${\bigl(f(p)\bigr)}_{p \leq N^{6/7}}$, and note that, when restricted to the above event, conditioning on ${\bigl(f(p)\bigr)}_{p \leq N^{6/7}}$ fixes the set $\tilde{\mathcal{A}}$. To prove the lower bound in the Steinhaus case, it suffices to show that for some small $\varepsilon > 0$ (which we assume to be sufficiently small throughout),
\[ \tilde{\p} \biggl( \max_{\theta \in \mathcal{\tilde{A}}} \frac{1}{\sqrt{N}} \Re \Bigl( \sum_{N^{6/7} < p \leq N} f(p) \sum_{m \leq N/p} f(m) \e (mp \theta) \Bigr) \leq (4/29 + \varepsilon) \sqrt{\log N} \biggr), \] 
is $o(1)$ with probability $1-o(1)$ over all realisations of ${\bigl(f(p)\bigr)}_{p \leq N^{6/7}}$. For the Rademacher case we replace the constant $\frac{4}{29}$ by $\frac{4 \sqrt{6}}{29 \pi}$.
Since the conditioning fixes the innermost sum, this is a maximum of a sum of independent random variables with coefficients $\sum_{m \leq N/p} f(m) \e (mp \theta)$. We approximate this sum by a multivariate normal random vector using Normal Approximation Result~\ref{na}, giving an upper bound for the above probability of
\begin{align*}
& \p \bigr( \max_{\theta \in \tilde{\mathcal{A}}} Z_\theta \leq (4/29 + \varepsilon) \sqrt{\log N} + \eta \bigl) + O \Biggl( \frac{1}{N^{3/2} \eta^3} \sum_{N^{6/7} < p \leq N} {\Bigl( \sum_{\theta \in \tilde{\mathcal{A}}} \Bigl| \sum_{m \leq N/p} f(m) \e (mp\theta) \Bigr| \Bigr)}^3 \Biggr)  \\ & + O \Biggl( \frac{1}{N \eta^2} \sum_{\theta_1, \theta_2 \in \tilde{\mathcal{A}}} \sqrt{\sum_{N^{6/7} < p \leq N}\Bigl| \sum_{m \leq N/p} f(m) \e (mp\theta) \Bigr|^2 \Bigl| \sum_{m \leq N/p} f(m) \e (m p \theta_2) \Bigr|^2} \Biggr), \label{equ:2.2} \tag{2.2}
\end{align*}
for any small $\eta > 0$. For $f$ Steinhaus, $Z_\theta$ are jointly normal random variables with means \\ $\E Z_\theta \coloneqq \tilde{\E} \frac{1}{\sqrt{N}} \Re \sum_{N^{6/7} < p \leq N} f(p) \sum_{m \leq N/p} f(m) \e (mp \theta) = 0  $, variances
\begin{align*}
\E Z_\theta^2 & \coloneqq \tilde{\E} {\Bigl( \frac{1}{\sqrt{N}} \Re \sum_{N^{6/7} < p \leq N} f(p) \sum_{m \leq N/p} f(m) \e (mp \theta) \Bigr)}^2 = \frac{1}{2N} \sum_{N^{6/7} < p \leq N} \Bigl| \sum_{m \leq N/p} f(m) \e (mp \theta) \Bigr|^2,
\end{align*}
and covariances $\E Z_{\theta_1} Z_{\theta_2} \coloneqq $
\begin{align*}
&\tilde{\E} \Bigl( \frac{1}{\sqrt{N}} \Re \sum_{N^{6/7} < p \leq N} f(p) \sum_{m \leq N/p} f(m) \e (mp \theta_1) \Bigr) \Bigl( \frac{1}{\sqrt{N}} \Re \sum_{N^{6/7} < p \leq N} f(p) \sum_{m \leq N/p} f(m) \e (mp \theta_2) \Bigr) \\
 =& \, \frac{1}{2N} \Re \Bigl( \sum_{N^{6/7} < p \leq N} \sum_{m_1 \leq N/p} \sum_{m_2 \leq N/p} f(m_1) \conj{f(m_2)} \e \bigl(p(m_1 \theta_1 - m_2 \theta_2)\bigr) \Bigr),
\end{align*}
for $\theta_1, \theta_2 \in \tilde{\mathcal{A}}$, $\theta_1 \neq \theta_2$. In the case of $f$ Rademacher, we again have zero mean. However, since $f$ takes only real values, we have variances
\[ \E Z_\theta^2 = \frac{1}{N} \sum_{N^{6/7} < p \leq N} {\Bigl( \Re \sum_{m \leq N/p} f(m) e(m p \theta) \Bigr)}^2 ,\]
and covariances
\[ \E Z_{\theta_1} Z_{\theta_2} = \frac{1}{N} \sum_{N^{6/7} < p \leq N} \Bigl( \Re \sum_{m_1 \leq N/p} f(m_1) \e (m_1 p \theta_1) \Bigr) \Bigl( \Re \sum_{m_2 \leq N/p} f(m_2) \e (m_2 p \theta_2) \Bigr) ,\]
instead. We proceed by showing that, in each case, all terms in~\eqref{equ:2.2} are $o(1)$ asymptotically almost surely.
\subsection{The `big Oh' terms are small}
Here we show that the `big Oh' terms in~\eqref{equ:2.2} are small asymptotically almost surely. To do this we will apply Lemma~\ref{L:Expectation}, which allows us to handle both the Rademacher and Steinhaus cases. By Markov's inequality for first moments, followed by Hölder's inequality and Lemma~\ref{L:Expectation}, the probability that the first `big Oh' term in~\eqref{equ:2.2} is larger than $1/\log N$ is
\begin{align*}
& \ll \frac{\log N}{\eta^3 N^{3/2}} \sum_{N^{6/7} < p \leq N} \E {\Bigl( \sum_{\theta \in \tilde{\mathcal{A}}} \Bigl| \sum_{m \leq N/p} f(m) \e (mp \theta) \Bigr| \Bigr)}^3 \\
& \leq \frac{(\log N)|\mathcal{A}|^2}{\eta^3 N^{3/2}} \sum_{N^{6/7} < p \leq N} \sum_{\theta \in \mathcal{A}} \E \Bigl| \sum_{m \leq N/p} f(m) \e (mp \theta) \Bigr|^3 \\
& \ll \frac{\log N |\mathcal{A}|^2}{\eta^3 N^{3/2}} \sum_{N^{6/7} < p \leq N} \sum_{\theta \in \mathcal{A}} {\biggl(\frac{N}{p} {(\log N)}^2 \biggr)}^{3/2} \ll \frac{{(\log N)}^3}{\eta^3 N^{3/56}} .
\end{align*}
Here we have used the inequalities $|\mathcal{A}| \ll N^{1/8}$ and $ \sum_{n \leq N} \tau_k (n) \ll_k N {(\log N)}^{k-1}$ (see for example~\cite[Lemma~3.1]{BNR}). Similarly, for the second `big Oh' term in~\eqref{equ:2.3}, we apply Cauchy--Schwarz, followed by Lemma~\ref{L:Expectation}, to find that the probability of being larger than $1/\log N$ is
\begin{align*}
&\ll \frac{\log N}{N \eta^2} \sum_{\theta_1, \theta_2 \in \mathcal{A}} \sqrt{\sum_{N^{6/7} < p \leq N} \sqrt{\E \Bigl| \sum_{m \leq N/p} f(m) \e (mp \theta_1 ) \Bigr|^4 \E \Bigl| \sum_{m \leq N/p} f(m) \e (mp \theta_2 ) \Bigr|^4 }} \\
&\leq \frac{\log N}{N \eta^2} \sum_{\theta_1, \theta_2 \in \mathcal{A}} \sqrt{\sum_{N^{6/7} < p \leq N} {\Bigl( \sum_{m \leq N/p} \tau_3 (m) \Bigr)}^2} \ll \frac{{(\log N)}^3 |\mathcal{A}|^2}{\eta^2} \sqrt{\sum_{N^{6/7} < p \leq N} \frac{1}{p^2}} \ll \frac{{(\log N)}^{5/2}}{\eta^2 N^{5/28}}.
\end{align*}
We conclude that our ``big Oh'' terms are $\ll 1/\log N$ on a set of probability $1 - O \Bigl( \frac{{(\log N)}^{5/2}}{\eta^2 N^{5/28}} \Bigr) - O \Bigl(\frac{{(\log N)}^3}{\eta^3 N^{3/56}} \Bigr)$ over realisations of ${\bigl(f(p)\bigr)}_{p \leq N^{6/7}}$. Further restricting ourselves to these events (that depend only on ${(f(p))}_{p \leq N^{6/7}}$), it suffices to bound $\p \bigr( \max_{\theta \in \tilde{\mathcal{A}}} Z_\theta \leq (4/29 + \varepsilon) \sqrt{\log N} + \eta \bigl)$ in the Steinhaus case, and $\p \bigr( \max_{\theta \in \tilde{\mathcal{A}}} Z_\theta \leq (4\sqrt{6}/29\pi + \varepsilon) \sqrt{\log N} + \eta \bigl)$ in the Rademacher case.
\subsection{Variance and covariance estimates}
In this section we examine the size of the variances $\E Z_\theta^2$ and show that the covariances $\E Z_{\theta_1} Z_{\theta_2}$ are small, thus allowing the use of Normal Comparison Result~\ref{nc} to estimate the size of the first term in~\eqref{equ:2.2}. Recall that in the Steinhaus case, our variances are \[\E Z_\theta^2 = \frac{1}{2N} \sum_{N^{6/7} < p \leq N} \Bigl| \sum_{m \leq N/p} f(m) \e (mp \theta) \Bigr|^2 ,\]
for $\theta \in \tilde{\mathcal{A}}$. We first expand out the square. Separating out the diagonal term and changing the order of summations in the off-diagonal term yields
\begin{align*}
&\frac{1}{2N} \sum_{N^{6/7} < p \leq N} \Bigl( \frac{N}{p} + O(1) \Bigr) \\
+&\frac{1}{2N} \sum_{m_1 < N^{1/7}} f(m_1) \sum_{\substack{m_2 < N^{1/7} \\ m_1 \neq m_2}} \conj{f(m_2)} \sum_{N^{6/7} < p \leq \min \{ N/m_1, N/m_2\}} \e (p \theta(m_1 - m_2)) . 
\end{align*}
It follows from Mertens' second theorem that the first term is $\frac{\log (7/6)}{2} + o(1)$. For the second term, note that the innermost sum is
\begin{align*}
\Biggl( \sum_{N^{6/7} < n \leq \min\{ N/m_1, N/m_2\}} \frac{\Lambda(n)}{\log n} \e (n \theta(m_1 - m_2)) \Biggr) + O \Bigl( \sqrt{\min \{N/m_1, N/m_2\}} \Bigr),
\end{align*}
so applying the triangle inequality we find that the second term is
\begin{align*}
\ll & \frac{1}{N} \sum_{m_1 < N^{1/7}} \sum_{m_2 < m_1} \Bigl| \sum_{N^{6/7} < n \leq N/m_1} \frac{\Lambda(n)}{\log n} \e (n \theta (m_1 - m_2)) \Bigr| + O \Bigl( \frac{1}{N^{2/7}} \Bigr) \\
\ll & \frac{1}{N \log N} \sum_{m_1 < N^{1/7}} \sum_{m_2 < m_1} \max_{N^{6/7} < X \leq N/m_1} \Bigl| \sum_{N^{6/7} < n \leq X} \Lambda(n) \e (n \theta (m_1 - m_2)) \Bigr| + O \Bigl( \frac{1}{N^{2/7}} \Bigr),
\end{align*}
where the second line follows from Abel summation.
We bound this main term using Exponential Sum Result~\ref{es} with $q \asymp N^{1/2}$, seeing as $\theta$ will be of the from $p / q$ according to~\eqref{equ:2.1}, and certainly $ (p (m_1 - m_2), q) = 1$ for any permitted choice of $p, m_1, m_2$. This gives the bound
\begin{align*}
& \ll \frac{{(\log N)}^3}{N} \sum_{m_1 < N^{1/7}} \sum_{m_2 < m_1} \biggl( \biggl( \frac{N}{m_1} \biggr) N^{-1/4} + {\biggl( \frac{N}{m_1} \biggr)}^{3/4} + {\biggl( \frac{N}{m_1} \biggr)}^{1/2} N^{1/4} \biggr) \ll \frac{{(\log N)}^3}{N^{1/14}}.
\end{align*}
Since $\log (7/6)$ is only slightly larger than $2/13$, we have the variance bound
\begin{equation}
\E Z_\theta^2 \geq \frac{1}{13} + o(1) \label{equ:2.3} \tag{2.3},
\end{equation}
uniformly for $\theta \in \tilde{\mathcal{A}}$, which will be sufficient for our purposes. In the Rademacher case, we proceed similarly with some minor alterations. First of all, the contribution from diagonal terms is $\frac{1}{2N} \sum_{N^{6/7} < p \leq N} \frac{6}{\pi^2} \frac{N}{p} + O \bigl( \sqrt{N/p} \bigr)$, since the Rademacher random multiplicative function is supported only on squarefree integers. This gives a contribution of $\frac{6}{13 \pi^2} + O \bigl( \frac{1}{\log N} \bigr)$ from the diagonal terms. The remaining terms are \begin{align*} & \frac{1}{4N} \sum_{N^{6/7} < p \leq N} \sum_{m_1 \leq N/p} \sum_{\substack{m_2 \leq N/p \\ m_2 \neq m_1}} f(m_1) f(m_2) \Bigl( \e \bigl((m_1 - m_2) p \theta \bigr) + \e \bigl((m_2 - m_1) p \theta \bigr) \Bigr) \\ 
+ & \frac{1}{4N} \sum_{N^{6/7} < p \leq N} \sum_{m_1 \leq N/p} \sum_{m_2 \leq N/p} f(m_1) f(m_2) \Bigl( \e \bigl((m_1 + m_2) p \theta \bigr) + \bigl((- m_1 - m_2) p \theta \bigr)\Bigr). \end{align*} 
All these terms can be shown to be negligible (i.e. $o(1)$) using the same arguments as in the Steinhaus case, giving $\E Z_\theta^2 \geq \frac{6}{13 \pi^2} + o(1)$ in the Rademacher case. \\ We now move on to handling covariances in the Steinhaus case. Recall that these are
\[ \E Z_{\theta_1} Z_{\theta_2} = \frac{1}{2N} \Re \Bigl( \sum_{N^{6/7} < p \leq N} \sum_{m_1 \leq N/p} \sum_{m_2 \leq N/p} f(m_1) \conj{f(m_2)} \e \bigl(p(m_1 \theta_1 - m_2 \theta_2)\bigr) \Bigr),\]
for $\theta_1, \theta_2 \in \tilde{\mathcal{A}}$, $\theta_1 \neq \theta_2$. Swapping the order of summation and applying the triangle inequality yields an upper bound
\begin{align*}
& \, \frac{1}{2N} \sum_{m_1 < N^{1/7}} \sum_{m_2 < N^{1/7}} \Bigl| \sum_{N^{6/7} < p \leq \min\{N/m_1, N/m_2\}} \e \bigl( p(m_1 \theta_1 - m_2 \theta_2) \bigr) \Bigr| \\
= & \, \frac{1}{2N} \sum_{m_1 < N^{1/7}} \sum_{m_2 < N^{1/7}} \Bigl| \sum_{N^{6/7} < n \leq \min\{N/m_1, N/m_2\}} \frac{\Lambda(n)}{\log n} \e \bigl( n(m_1 \theta_1 - m_2 \theta_2) \bigr) \Bigr| + O \Bigl( \frac{1}{N^{2/7}} \Bigr), \end{align*}
and by Abel summation, this main term is
\begin{align*}
\ll & \, \frac{1}{N \log N} \sum_{m_1 < N^{1/7}} \sum_{m_2 < N^{1/7}} \max_{N^{6/7} < X \leq \min\{N/m_1, N/m_2\}}\Bigl| \sum_{N^{6/7} < n \leq X} \Lambda(n) \e \bigl( n(m_1 \theta_1 - m_2 \theta_2) \bigr) \Bigr|,
\end{align*}
Applying Exponential Sum Result~\ref{es} similarly to above, we find that this is
\begin{align*}
& \ll \frac{{(\log N)}^3}{N} \sum_{m_1 < N^{1/7}} \sum_{m_2 < m_1} \biggl( \biggl( \frac{N}{m_1} \biggr) N^{-1/4} + {\biggl( \frac{N}{m_1} \biggr)}^{3/4} + {\biggl( \frac{N}{m_1} \biggr)}^{1/2} N^{1/4} \biggr) \ll  \frac{{(\log N)}^3}{N^{1/14}},
\end{align*}
uniformly in $\theta_1, \theta_2$, and so our covariances are all $o(1)$ in the Steinhaus case. The covariances in the Rademacher case can be written as $\E Z_{\theta_1} Z_{\theta_2} = $
\[ \frac{1}{2N} \Re \biggl( \sum_{N^{6/7} < p \leq N}  \sum_{m_1 \leq N/p} \sum_{m_2 \leq N/p} f(m_1) f(m_2) \Bigl( \e \bigl( p (m_1 \theta_1 + m_2 \theta_2) \bigr) + \e \bigl( p (m_1 \theta_1 - m_2 \theta_2) \bigr) \Bigr) \biggr). \]
Applying the same argument as in the Steinhaus case, we find that the covariances in the Rademacher case are also $o(1)$, uniformly in $\theta_1, \theta_2$.
\subsection{Combining results}
From~\eqref{equ:2.2} and the preceeding estimates we conclude that, in the Steinhaus case, for any fixed small $\varepsilon > 0$,
\begin{align*} \p \Bigl( \max_{\theta \in [0,1]} |P_N (\theta) | \leq 4/29 \sqrt{\log N} \Bigr) \leq& \p \Bigl( \max_{\theta \in \tilde{\mathcal{A}}} Z_\theta \leq (4/29 + \varepsilon)\sqrt{\log N} + \eta \Bigr) + O\biggl( \frac{1}{{(\log N)}^{1/5}} \biggr)  \\ & + O \biggl( \frac{{(\log N)}^{5/2}}{\eta^2 N^{5/28}} \biggr) + O \biggl( \frac{{(\log N)}^3}{\eta^3 N^{3/56}} \biggr), \end{align*}
where $\tilde{\mathcal{A}}$ is a set of size $\geq 0.99 N^{1/8}$ and $Z_\theta$ is a multivariate Gaussian with variances $\E Z_\theta^2 \geq \frac{1}{13} + o(1)$ for any $\theta \in \tilde{\mathcal{A}}$ by~\eqref{equ:2.3} and covariances $\E Z_{\theta_1} Z_{\theta_2} = o(1)$ for any $\theta_1 \neq \theta_2 \in \tilde{\mathcal{A}}$. Taking $\eta$ constant and $\varepsilon>0$ sufficiently small, in the Steinhaus case we have
\begin{align*} \p \bigr( \max_{\theta \in \tilde{\mathcal{A}}} Z_\theta &\leq (4/29 + \varepsilon) \sqrt{\log N} + \eta \bigl) \leq \p \Bigl( \max_{\theta \in \tilde{\mathcal{A}}} \frac{Z_\theta}{\sqrt{\E Z_{\theta}^2}} \leq \Bigl( \frac{4\sqrt{13}}{29} + 4\varepsilon \Bigr) \sqrt{\log N} \Bigr),
\end{align*}
for $N$ sufficiently large. Since $|\tilde{\mathcal{A}}| \geq 0.99N^{1/8}$, this is
\begin{align*}
& \leq \p \Bigl( \max_{\theta \in \tilde{\mathcal{A}}} \frac{Z_\theta}{\sqrt{\E Z_{\theta}^2}} \leq \Bigl( \frac{4\sqrt{13}}{29} + 4\varepsilon \Bigr) \sqrt{(8 + \varepsilon)\log |\tilde{A}|} \Bigr) \\ &\leq \p \biggl( \max_{\theta \in \tilde{\mathcal{A}}} \frac{Z_\theta}{\sqrt{\E Z_{\theta}^2}} \leq \sqrt{(1664/841 + 50 \varepsilon)\log |\tilde{A}|} \biggr) ,
\end{align*}
when $\varepsilon > 0$ is sufficiently small. Seeing as our covariances are each $o(1)$, Theorem~\ref{T:lb} is concluded by applying Normal Comparison Result~\ref{nc} to deduce that this probability is $o(1)$. In the Rademacher case we replace $4/29$ by the smaller constant $4 \sqrt{6}/29 \pi$, since we have variances of size atleast $6/13 \pi^2 + o(1)$. We remark that the different bounds are a consequence of the support of our random multiplicative functions.
\end{proof} 
\section{Upper Bound}
\subsection{Conditional variance bound}\label{s:cvb}
The proof of the upper bound revolves around proving the following proposition on the conditional variance of our sum. We recall from section~\ref{sb:uboutline} that it is sufficient to take the maximum over some sufficiently dense set $\mathcal{D}$.
\begin{proposition}\label{p:mainprop}
Let $f$ be a Steinhaus random multiplicative function and \[ \mathcal{D} = \biggl\{ \frac{a}{q} + \frac{j}{4 \pi N} : q \leq 4 \pi \sqrt{N}, \, 1 \leq a \leq q \, \, \mathrm{ with } \, \, (a,q)=1, \, j \in \biggl[\frac{-\sqrt{N}}{q}, \frac{\sqrt{N}}{q} \biggr] \cap \mathbb{Z} \biggr\}. \]
For any $\varepsilon > 0$, with probability $1 - o_{\varepsilon} (1)$, we have
\[ \max_{\theta \in \mathcal{D}} \frac{1}{N} \sum_{N^{0.8} \leq p \leq N} \Bigl| \sum_{m \leq N/p} f(m) \e (mp \theta) \Bigr|^2 \ll {(\log N)}^{5/2 + \varepsilon} .\]
\end{proposition}
The proposition will follow from the proceeding lemmas. Here and throughout, for $\theta \in \mathcal{D}$, we write $\theta = \frac{a}{q} + \frac{j}{4 \pi N}$ for some $q \leq 4 \pi \sqrt{N}$, $1 \leq a \leq q$ with $(a,q)=1$ and $j \in \bigl[\frac{-\sqrt{N}}{q}, \frac{\sqrt{N}}{q}\bigr] \cap \Z$.
\begin{lemma}\label{l:novariation}
For $f$ Steinhaus, for any $\varepsilon > 0$, with probability $1-o_{\varepsilon} (1)$, we have
\[ \max_{q \leq 4 \pi \sqrt{N}} \sum_{r \leq N^{0.2}} \sum_{(b,q)=1} \frac{1}{r(r+1)\varphi(q) \log N} \Bigl| \sum_{m \leq r} f(m) \e \Bigl( \frac{m b}{q} \Bigr)  \Bigr|^2 \ll {( \log N )}^{3/2 + \varepsilon} . \]   
\end{lemma}
\begin{lemma}\label{l:variation}
For $f$ Steinhaus, for any $\varepsilon > 0$, with probability $1-o_{\varepsilon} (1)$,
\[ \max_{\substack{q \leq 4 \pi \sqrt{N} \\ 1 \leq |j| \leq \frac{\sqrt{N}}{q}}} \sum_{r \leq N^{0.2}} \sum_{(b,q)=1} \sum_{k=0}^{\frac{j N^{2/5}}{r(r+1)}} \frac{1}{N^{2/5} |j| \varphi(q) \log N } \Bigl| \sum_{m \leq r} f(m) \e \Bigl( m \Bigl( \frac{b}{q} + \frac{j}{4 \pi (r+1)} + \frac{k}{4 \pi N^{2/5}} \Bigr) \Bigr)  \Bigr|^2, \]
is $\ll {( \log N )}^{5/2 + \varepsilon}$.
\end{lemma}
When $j$ is positive, the sum over $k$ should be taken up to the largest integer below $\frac{j N^{2/5}}{r(r+1)}$. Note that the upper range of summation can be $<1$ only when $|j|=1$ and $r = \lfloor N^{0.2} \rfloor$. This specific case can be bounded using trivial bounds uniformly over $q \leq 4 \pi \sqrt{N}$, so one should not worry about it causing any issues. In Lemma~\ref{l:variation}, the sum over $k$ should be interpreted as taking negative values when $j$ is negative, in which case we sum up to the smallest integer above $\frac{j N^{2/5}}{r(r+1)}$: by convention we take this sum to be positive, despite the upper range being below the lower range. We maintain this convention throughout this section, and use $\sum_{(b,q)=1}$ as shorthand for the sum over $\{ 1 \leq b \leq q : \, (b,q)=1\}$.
\begin{proof}[Proof of Proposition~\ref{p:mainprop}, assuming Lemmas~\ref{l:novariation} and~\ref{l:variation}.]
Let $\theta = \frac{a}{q} + \frac{j}{4 \pi N}$. As mentioned in the outline of the proof (section~\ref{sb:uboutline}), we write the sum $\frac{1}{N} \sum_{N^{0.8} < p \leq N} \bigl| \sum_{m \leq N/p} f(m) \e (mp \theta) \bigr|^2$ as \begin{equation} \frac{1}{N} \sum_{r \leq N^{0.2}} \sum_{N/(r+1) < p \leq N/r} \Bigl| \sum_{m \leq r} f(m) \e (mp \theta) \Bigr|^2 . \label{equ:3.1} \tag{3.1} \end{equation} We first handle the case where $j \neq 0$. If $j > 0$, we bound the above by
\[
\frac{1}{N} \sum_{r \leq N^{0.2}} \sum_{k=0}^{\frac{j N^{2/5}}{r(r+1)}} \sum_{\frac{N}{r+1} + \frac{k N^{3/5}}{j} < p \leq \frac{N}{r+1} + \frac{(k+1) N^{3/5}}{j} } \Bigl| \sum_{m \leq r} f(m) \e \Bigl( m p \Bigl( \frac{a}{q} + \frac{j}{4 \pi N} \Bigr) \Bigr)  \Bigr|^2 .
\]
In the case $j < 0$, we replace $k+1$ by $k-1$ in the upper bound on the summation over $p$, and the analogous change should also be made in the calculations that follow. \\ Define $s^{(r,j)}_k = \frac{N}{r+1} + \frac{k N^{3/5}}{j}$. Now, approximating $\frac{p j}{4 \pi N}$ by elements of the form $\frac{j}{4 \pi (r+1)} + \frac{k}{4 \pi N^{2/5}} $ and using the fact that $|A+B|^2 \ll |A|^2 + |B|^2$, we find that the above display is
\[
\ll \frac{1}{N} \sum_{r \leq N^{0.2}} \sum_{k=0}^{\frac{j N^{2/5}}{r(r+1)}} \sum_{s^{(r,j)}_k < p \leq s^{(r,j)}_{k+1}} \Bigl| \sum_{m \leq r} f(m) \e \Bigl( m \Bigl( \frac{p a}{q} + \frac{j}{4 \pi (r+1)} + \frac{k}{4 \pi N^{2/5}} \Bigr) \Bigr)  \Bigr|^2 ,
\]
plus an error of size
\begin{align*}
\ll \frac{1}{N} \sum_{r \leq N^{0.2}} \sum_{N/(r+1) < p \leq N/r} \Bigl| \sum_{m \leq r} \frac{m}{N^{2/5}} \Bigr|^2 \ll \frac{1}{N^{9/5}} \sum_{r \leq N^{0.2}} \sum_{N/(r+1) < p \leq N/r} r^4 = o(1). 
\end{align*} 
Importantly, the size of this error term is uniform over all possible $\theta = \frac{a}{q} + \frac{j}{4 \pi N}$, and it is negligible in the context of Proposition~\ref{p:mainprop}. Seeing as $p a \Mod{q}$ will uniformly distribute over residue classes $\Mod{q}$ when $p$ runs over a sufficiently large interval, we write this main contribution as
\[
\frac{1}{N} \sum_{r \leq N^{0.2}} \sum_{(b,q)=1} \sum_{k=0}^{\frac{j N^{2/5}}{r(r+1)}}  \tilde{C}_{r, \theta, k, b} \Bigl| \sum_{m \leq r} f(m) \e \Bigl( m \Bigl( \frac{b}{q} + \frac{j}{4 \pi (r+1)} + \frac{k}{4 \pi N^{2/5}} \Bigr) \Bigr)  \Bigr|^2
\]
where \[ \tilde{C}_{r, \theta, k, b} = \# \Bigl\{ s^{(r,j)}_k  < p \leq s^{(r,j)}_{k+1} : \, pa \equiv b \Mod q \Bigr\}. \]
Applying the Brun--Titchmarsh bound for primes in arithmetic progressions gives $\tilde{C}_{r, \theta, k, b} \leq \frac{2 N^{3/5}}{|j| \varphi(q) \log ( N^{3/5}/ q |j| )}$. We remark that the bound only holds when $q$ is smaller than $N^{3/5}/|j|$, and seeing as $q|j| \leq \sqrt{N}$, this is certainly satisfied. Therefore, for $j \neq 0$, we obtain an upper bound for~\eqref{equ:3.1} of
\[
\ll \sum_{r \leq N^{0.2}} \sum_{(b,q)=1} \sum_{k=0}^{\frac{j N^{2/5}}{r(r+1)}} \frac{1}{N^{2/5} |j| \varphi(q) \log N } \Bigl| \sum_{m \leq r} f(m) \e \Bigl( m \Bigl( \frac{b}{q} + \frac{j}{4 \pi (r+1)} + \frac{k}{4 \pi N^{2/5}} \Bigr) \Bigr)  \Bigr|^2, 
\]
plus a $o(1)$ term. We note that our sum no longer depends on the parameter $a$ (where $\theta = \frac{a}{q} + \frac{j}{4 \pi N}$). Seeing as the implicit constant is uniform over $\theta$, the maximum of our sum over $q \leq 4 \pi \sqrt{N}, 1 \leq |j| \leq \frac{\sqrt{N}}{q}$ can be bounded using Lemma~\ref{l:variation}. In the case where $j=0$ (i.e. $\theta = \frac{a}{q}$ for some $q \leq 4 \pi \sqrt{N}$ and $(a,q)=1$), we can start at~\eqref{equ:3.1} and immediately apply the Brun--Titichmarsh bound similarly. Having done this, one finds that the maximum of the sum in Proposition~\ref{p:mainprop} over $\theta = \frac{a}{q}$ with $q \leq 4 \pi \sqrt{N}$, $(a,q)=1$ can be bounded above using Lemma~\ref{l:novariation}. 
\end{proof}
\begin{remark}\label{r:optimising}
\normalfont~For a slightly improved result, one could instead approximate $\frac{p j}{4 \pi N}$ by points of the form $\frac{j}{4 \pi (r+1)} + \frac{100 k}{4 \pi N^{3/7}} $, say, and take $Q = \frac{N^{3/7}}{2}$ (see Remark~\ref{r:interplay}). This would allow one to upper bound the sum $\sum_{n \leq N, \, P(n) \geq N^{5/7} } f(n) \e (n \theta) $ instead. This alteration makes Lemma~\ref{l:novariation} more technical, so these improvements have not been implemented here.
\end{remark}
We proceed with proving Lemmas~\ref{l:novariation} and~\ref{l:variation}. Once this is done, a short application of Proposition~\ref{p:mainprop} will complete the proof.
\begin{proof}[Proof of Lemma~\ref{l:novariation}]
To begin the proof, one may immediately try to apply a union bound over $q \leq 4 \pi \sqrt{N}$, which will be efficient assuming that there are few realisations of $f$ for which the sum is large over many values of $q$. However, regardless of $f$, and the value of $q$, our variance has a significant deterministic contribution from the diagonal term in the innermost sum $\bigl| \sum_{m \leq r} f(m) \e (m \alpha)  \bigr|^2$, and taking a union bound over it would be inefficient. This contribution reflects the mean of our conditional variance. We begin by separating it out, writing our innermost sum in Lemma~\ref{l:novariation} as
\[ \sum_{m_1 \leq r} \sum_{m_2 \leq r} f(m_1) \conj{f(m_2)} \e \biggl( \frac{b(m_1 - m_2 )}{q} \biggr). \]
The case where $m_1 = m_2$ contributes
\[ \ll \sum_{r \leq N^{0.2}} \sum_{(b,q)=1} \frac{1}{(r+1) \varphi(q) \log N} \ll \sum_{r \leq N^{0.2}} \frac{1}{(r+1) \log N} \ll  1, \]
to the full sum, regardless of $q$. We can disregard this diagonal contribution seeing as it is $\ll {(\log N)}^{3/2 + \varepsilon}$. One may interpret what follows as calculating the variance of our conditional variance, essentially equating to a fourth moment bound. For simplicity, we split the sums over $r$ into dyadic ranges, letting $r \sim M$ denote the range $M \leq r \leq 2M$. We will calculate second moments of this sum, and it will be helpful to use the fact that the duplicated sums over $r$ both run over the same orders of magnitude. To prove Lemma~\ref{l:novariation}, it suffices to prove that, for any $M \leq \frac{N^{0.2}}{2}$, the probability that
\[ \max_{q \leq 4 \pi \sqrt{N}} \sum_{r \sim M} \sum_{(b,q)=1} \frac{1}{r(r+1)\varphi(q) \log N} \sum_{\substack{m_1, m_2 \leq r \\ m_1 \neq m_2}} f(m_1) \conj{f(m_2)} \e \biggl( \frac{b( m_1 - m_2)}{q} \biggr)  \]
is larger than ${(\log N)}^{1/2 + \varepsilon}$ is $o_{\varepsilon}(\frac{1}{\log N})$. Note the changes by factors of $\log N$ due to taking a union bound over dyadically spaced $M$. To bound this probability, we apply the union bound followed by Markov's inequality for second moments, arriving at the upper bound
\begin{multline*} \ll \frac{1}{M^2 {{(\log N)}^{3 + 2 \varepsilon}}} \sum_{q \leq 4 \pi \sqrt{N}} \frac{1}{{\varphi(q)}^2} \sum_{\substack{m_1, m_2 \leq 2M, \, m_1 \neq m_2 \\ m_3, m_4 \leq 2M, \, m_3 \neq m_4 \\ m_1 m_3 = m_2 m_4}} \\ \biggl| \sum_{(b_1, q)=1} \e \biggl(\frac{b_1(m_1 - m_2)}{q}\biggr) \sum_{(b_2, q)=1} \e \biggl(\frac{b_2(m_3 - m_4)}{q}\biggr) \biggr| , \end{multline*}
after rearranging and applying the triangle inequality. Here we have used the fact that $\sum_{\max{ \{m_1, m_2, M \} } \leq r_1 \leq 2M } \frac{1}{r_1 (r_1 + 1)} \ll \frac{1}{M}$, and likewise for the sum over $r_2$. We notice that the sums over $b_1$ and $b_2$ are Ramanujan sums $c_q (n) = \sum_{(b,q)=1} \e (bn/q)$, so we apply the bound $|c_q (n)| \leq (|n|,q)$ (see, for example, exercise 3~\citet[Section~4.1.1]{MVmultnt}). This gives an upper bound for the previous display of
\begin{align*} 
\ll & \frac{1}{M^2 {{(\log N)}^{3 + 2 \varepsilon}}} \sum_{q \leq 4 \pi \sqrt{N}} \frac{1}{{\varphi(q)}^2} \sum_{\substack{m_1, m_2 \leq 2M, \, m_1 \neq m_2 \\ m_3, m_4 \leq 2M, \, m_3 \neq m_4 \\ m_1 m_3 = m_2 m_4}} (|m_1 - m_2|, q) (|m_3 - m_4|, q),
\end{align*}
and we need to show that this is $o_\varepsilon \bigl( \frac{1}{\log N} \bigr)$. We deal with the greatest common divisor terms by rewriting them as $\sum_{d| q, \, |m_1 - m_2| } \varphi (d) \leq \sum_{d| q, \, |m_1 - m_2| } d $ (and similarly for $m_3, m_4$). By doing this, we get an upper bound for our probability of
\begin{align*} 
\ll & \frac{1}{M^2 {{(\log N)}^{3 + 2 \varepsilon}}} \sum_{d_1, d_2 \leq 2M} d_1 d_2 \sum_{\substack{q \leq 4 \pi \sqrt{N} \\ q \equiv 0 \Mod [d_1,d_2]}} \frac{1}{{\varphi(q)}^2} \sum_{\substack{m_1, m_2 \leq 2M, \, m_1 \neq m_2 \\ m_3, m_4 \leq 2M, \, m_3 \neq m_4 \\ m_1 m_3 = m_2 m_4 \\ m_1 \equiv m_2 \Mod d_1 \\ m_3 \equiv m_4 \Mod d_2}} 1.
\end{align*}
We proceed by rewriting the quadruple sum over $m_1, m_2, m_3$ and $m_4$ using a factorisation of Vaughan and Wooley (see~\cite[Lemma~2.3]{BNR}, or~\cite[Section~8]{VaughanWooley} for the original work).
Specifically, we let $m_1 = a_1 a_3$, $m_2 = a_1 a_4$, $m_3 = a_2 a_4$ and $m_4 = a_2 a_3$, subject to the conditions $a_1, a_2 \leq 2M$, $a_3, a_4 \leq \min \bigl\{ \frac{2M}{a_1}, \frac{2M}{a_2} \bigr\}$, $a_3 \neq a_4$ and $a_3 \equiv a_4 \Mod \bigl[ \frac{d_1}{(a_1, d_1)}, \frac{d_2}{(a_2, d_2)} \bigr]$. Furthermore, by symmetry, we can assume that $a_2 \leq a_1$. This gives the upper bound
\begin{align*}
& \ll \frac{1}{M^2 {{(\log N)}^{3 + 2 \varepsilon}}} \sum_{d_1, d_2 \leq 2M} d_1 d_2 \sum_{\substack{q \leq 4 \pi \sqrt{N} \\ q \equiv 0 \Mod [d_1,d_2]}} \frac{1}{{\varphi(q)}^2} \sum_{a_2 \leq a_1 \leq 2M} \sum_{\substack{a_3, a_4 \leq \frac{2M}{a_1}, a_3 \neq a_4 \\ a_3 \equiv a_4 \Mod \bigl[ \frac{d_1}{(a_1, d_1)}, \frac{d_2}{(a_2, d_2)} \bigr]}} 1 \\
& \ll \frac{{(\log \log N)}^2}{{(\log N)}^{3 + 2 \varepsilon}} \sum_{d_1, d_2 \leq 2M} \frac{d_1 d_2}{{[d_1, d_2]}^2} \sum_{a_2 \leq a_1 \leq 2M} \frac{1}{a_1^2 \bigl[ \frac{d_1}{(a_1, d_1)}, \frac{d_2}{(a_2, d_2)} \bigr]} , \label{equ:3.2} \tag{3.2}
\end{align*}
where we have used the fact that $\varphi(n) \gg n/{\log \log n} $ to perform the sum over $q$. To proceed, we let $g_1 = (a_1, d_1)$, $D_1 = \frac{d_1}{g_1}$ and $A_1 = \frac{a_1}{g_1}$. These imply that $(D_1, A_1)=1$, but we can drop this condition without much loss. We define $g_2, D_2$ and $A_2$ analogously. Our variables satisfy the bounds $g_1 \leq 2M$, $D_1 \leq \frac{2M}{g_1}$, $A_1 \leq \frac{2M}{g_1}$, $g_2 \leq A_1 g_1$, $D_2 \leq \frac{2M}{g_2}$ and $A_2 \leq \frac{A_1 g_1}{g_2}$, where we have used the fact that $a_2 \leq a_1$. One may then deduce the upper bound
\begin{align*}
\ll \frac{{(\log \log N)}^2}{{{(\log N)}^{3 + 2\varepsilon}}} \sum_{g_1 \leq 2M} \frac{1}{g_1^{3}} \sum_{A_1 \leq \frac{2M}{g_1}} \frac{1}{A_1^{2}} \sum_{g_2 \leq A_1 g_1} \frac{1}{g_2} \sum_{A_2 \leq \frac{A_1 g_1}{g_2}} \sum_{\substack{D_1 \leq \frac{2M}{g_1} \\ D_2 \leq \frac{2M}{g_2}}} \frac{(D_1, D_2) {{(D_1 g_1, D_2 g_2)}^2}}{D_1^2 D_2^2}, \tag{3.3} \label{equ:3.3}
\end{align*}
for our probability. Performing the sum over $A_2$, and letting $h = (D_1, D_2) \leq \min\{\frac{2M}{g_1}, \frac{2M}{g_2}\}$, $E_1 = \frac{D_1}{h} \leq \frac{2M}{g_1 h}$ and $E_2 = \frac{D_2}{h} \leq \frac{2M}{g_2 h}$, this is
\begin{align*}
& \ll \frac{{(\log \log N)}^2}{{{(\log N)}^{3 + 2\varepsilon}}} \sum_{g_1 \leq 2M} \frac{1}{g_1^2} \sum_{A_1 \leq \frac{2M}{g_1}}  \frac{1}{A_1} \sum_{ g_2 \leq A_1 g_1} \frac{1}{g_2^2} \sum_{h \leq \min\{\frac{2M}{g_1}, \frac{2M}{g_2}\}} \frac{1}{h} \sum_{\substack{E_1 \leq \frac{2M}{g_1 h} \\ E_2 \leq \frac{2M}{g_2 h}}} \frac{{(E_1 g_1, E_2 g_2)}^2}{E_1^2 E_2^2}.
\end{align*}
Writing ${(E_1 g_1, E_2 g_2)}^2 \leq \sum_{R \mid E_1 g_1, E_2 g_2} R^2 $, we find that the innermost sum over $E_1$ and $E_2$ is $\ll \sum_{R \leq \frac{2M}{h}} \frac{{(R, g_1)}^2 {(R, g_2)}^2}{R^2}$. Therefore, the above is
\begin{align*}
& \ll \frac{{(\log \log N)}^2}{{{(\log N)}^{3 + 2\varepsilon}}} \sum_{g_1 \leq 2M} \frac{1}{g_1^2} \sum_{A_1 \leq \frac{2M}{g_1}}  \frac{1}{A_1} \sum_{ g_2 \leq A_1 g_1} \frac{1}{g_2^2} \sum_{h \leq \min\{\frac{2M}{g_1}, \frac{2M}{g_2}\}} \frac{1}{h} \sum_{R \leq 2M/h} \frac{{(R,g_1)}^2 {(R, g_2)}^2}{R^2} .
\end{align*}
Similarly, since ${(R, g_1)}^2 \leq \sum_{t_1 \mid R, g_1} t_1^2$ (and likewise for ${(R, g_2)}^2$), this is
\begin{align*}
& \ll \frac{{(\log \log N)}^2}{{{(\log N)}^{2 + 2\varepsilon}}} \sum_{t_1, t_2 \leq 2M} \frac{t_1^2 t_2^2}{{[t_1,t_2]}^2} \sum_{\substack{g_1 \leq 2M \\ g_1 \equiv 0 \Mod t_1}} \frac{1}{g_1^2} \sum_{A_1 \leq \frac{2M}{g_1}}  \frac{1}{A_1} \sum_{\substack{ g_2 \leq A_1 g_1 \\ g_2 \equiv 0 \Mod t_2}} \frac{1}{g_2^2} \ll \frac{{(\log \log N)}^2}{{{(\log N)}^{1 + 2\varepsilon}}},
\end{align*}
which is $o_{\varepsilon} \bigl( \frac{1}{\log N} \bigr)$, as required. 
\end{proof}
\begin{proof}[Proof of Lemma~\ref{l:variation}]
We now move on to the term
\[ \sum_{r \leq N^{0.2}} \sum_{(b,q)=1} \sum_{k=0}^{\frac{j N^{2/5}}{r(r+1)}} \frac{1}{N^{2/5} |j| \varphi(q) \log N } \Bigl| \sum_{m \leq r} f(m) \e \Bigl( m \Bigl( \frac{b}{q} + \frac{j}{4 \pi (r+1)} + \frac{k}{4 \pi N^{2/5}} \Bigr) \Bigr)  \Bigr|^2, \]
and we wish to bound the maximum over $q \leq 4 \pi \sqrt{N}$ and $1 \leq |j| \leq \frac{\sqrt{N}}{q}$. We similarly remove the diagonal contribution from the innermost sum, and assume that $j$ is non-negative (the case where it is negative can be handled analogously). To prove Lemma~\ref{l:variation}, we just need to show that the probability of the event 
\begin{align*} \max_{\substack{q \leq 4 \pi \sqrt{N} \\ 1 \leq j \leq \frac{\sqrt{N}}{q}}} & \Biggl[ \sum_{r \leq N^{0.2}} \sum_{(b,q)=1} \sum_{k=0}^{\frac{j N^{2/5}}{r(r+1)}} \frac{1}{N^{2/5} j \varphi(q) \log N} \\ & \sum_{\substack{m_1, m_2 \leq r \\ m_1 \neq m_2}} f(m_1) \conj{f(m_2)} \e \Bigl( (m_1 - m_2 ) \Bigl( \frac{b}{q} + \frac{j}{4 \pi (r+1)} + \frac{k}{4 \pi N^{2/5}} \Bigr) \Bigr) \Biggr] \geq {(\log N)}^{5/2 + \varepsilon},  \end{align*}
is $o_{\varepsilon} (1)$. Breaking the sum over $r$ into dyadic intervals, it suffices to show that the probability of the event
\begin{align*} \max_{\substack{q \leq 4 \pi \sqrt{N} \\ 1 \leq j \leq \frac{\sqrt{N}}{q}}} \Biggl[ & \sum_{r \sim M} \sum_{(b,q)=1} \sum_{k = 0}^{\frac{j N^{2/5}}{r(r+1)}} \frac{1}{N^{2/5} j \varphi(q) \log N} \\ & \sum_{\substack{m_1, m_2 \leq r \\ m_1 \neq m_2}} f(m_1) \conj{f(m_2)} \e \Bigl( (m_1 - m_2 ) \Bigl( \frac{b}{q} + \frac{j}{4 \pi (r+1)} + \frac{k}{4 \pi N^{2/5}} \Bigr) \Bigr) \Biggr] \geq {(\log N)}^{3/2 + \varepsilon} , \end{align*}
is $o_{\varepsilon} \bigl(\frac{1}{\log N}\bigr)$ for $M \leq \frac{N^{0.2}}{2}$. Similarly to before, we apply the union bound and Markov's inequality with second moments to obtain an upper bound for this probability of
\begin{align*}
&\frac{1}{{(\log N)}^{5 + 2 \varepsilon}}\sum_{q \leq 4 \pi \sqrt{N}} \sum_{j \leq \frac{\sqrt{N}}{q}} \frac{1}{N^{4/5} j^2 {\varphi(q)}^2} \sum_{\substack{m_1, m_2 \leq 2M \\ m_3, m_4 \leq 2M \\ m_1 \neq m_2, \, m_3 \neq m_4 \\ m_1 m_3 = m_2 m_4 }} \sum_{\substack{r_1, r_2 \sim M \\ m_1, m_2 \leq r_1 \\ m_3, m_4 \leq r_2}} \sum_{\substack{(b_1,q)=1\\(b_2,q)=1}} \sum_{k_1, k_2 = 0}^{\frac{j N^{2/5}}{r(r+1)}} \\
& \e \Bigl( (m_1 - m_2 ) \Bigl( \frac{b_1}{q} + \frac{j}{4 \pi (r_1+1)} + \frac{k_1}{4 \pi N^{2/5}} \Bigr) \Bigr) \e \Bigl( (m_3 - m_4 ) \Bigl( \frac{b_2}{q} + \frac{j}{4 \pi (r_2+1)} + \frac{k_2}{4 \pi N^{2/5}} \Bigr) \Bigr).
\end{align*}
Applying the triangle inequality, we obtain a bound for the summand in the sum over $m_1, m_2, m_3, m_4$ of
\begin{align*}
\Bigl| \sum_{(b_1,q)=1} \e \Bigl(\frac{b_1 (m_1 - m_2) }{q} \Bigr) \sum_{k_1 = 0}^{\frac{j N^{2/5}}{r(r+1)}} \e \Bigl( \frac{k_1 (m_1 - m_2)}{4 \pi N^{2/5}} \Bigr) \sum_{\max\{m_1,m_2, M\} \leq r_1 \leq 2M} \e \Bigl( \frac{j (m_1 - m_2)}{r_1 + 1} \Bigr) \Bigr| ,
\end{align*}
along with the corresponding sums involving $m_3$ and $m_4$. Again, the sums over $b_1$ and $b_2$ are Ramanujan sums and contribute $ \, \leq (|m_1 - m_2|,q) (|m_3 - m_4|,q)$. The sums over $k_1$ and $k_2$ are geometric, so we upper bound their contribution by
\begin{align*} \ll & \min \Biggl\{ \frac{j N^{2/5}}{M^2}, \frac{1}{\| \frac{m_1 - m_2}{N^{2/5}} \|} \Biggr\} \min \Biggl\{ \frac{j N^{2/5}}{M^2}, \frac{1}{\| \frac{m_3 - m_4}{N^{2/5}} \|} \Biggr\}. \end{align*}
For ease of notation, we let
\begin{align*}
\mathcal{S}^{(M)}_{\mathbf{m},j} \coloneqq \min \Biggl\{ \frac{j}{M^2}, \frac{1}{| m_1 - m_2 |} \Biggr\} \min \Biggl\{ \frac{j}{M^2}, \frac{1}{| m_3 - m_4 |} \Biggr\} , \end{align*}
so that the sums over $k_1$ and $k_2$ contribute $\ll N^{4/5} \mathcal{S}^{(M)}_{\mathbf{m},j}$. Now, by Exponential Sum Result~\ref{es2}, the sums over $r_1$ and $r_2$ can be bounded above by a constant times
\begin{align*}
\mathcal{T}^{(M)}_{\mathbf{m},j} \coloneqq \biggl( \frac{|\theta_1|^{1/2}}{M^{1/2}} + \frac{M^{3/2}}{|\theta_1|^{1/2}} \biggr) \biggl(\frac{|\theta_2|^{1/2}}{M^{1/2}} + \frac{M^{3/2}}{|\theta_2|^{1/2}} \biggr),
\end{align*}
where $\theta_1 = j (m_1 - m_2)$ and $\theta_2 = j (m_3 - m_4)$. Overall, these give an upper bound on our probability of
\begin{align*} 
\ll \frac{1}{{(\log N)}^{5 + 2 \varepsilon}} \sum_{q \leq 4 \pi \sqrt{N}} \frac{1}{{\varphi(q)}^2} \sum_{j \leq \frac{\sqrt{N}}{q}} \frac{1}{j^2} \sum_{\substack{m_1, m_2 \leq 2M \\ m_3, m_4 \leq 2M \\ m_1 \neq m_2, \, m_3 \neq m_4 \\ m_1 m_3 = m_2 m_4 }} (|m_1 - m_2|, q) (|m_3 - m_4|, q) \mathcal{S}^{(M)}_{\mathbf{m},j} \mathcal{T}^{(M)}_{\mathbf{m},j}.
\end{align*}
Writing $(n,q) \leq \sum_{d | n,q} d$ and changing the order of summation, this is bounded above by 
\begin{align*} 
\frac{1}{{(\log N)}^{5 + 2 \varepsilon}} \sum_{d_1, d_2 \leq 2M} d_1 d_2 \sum_{\substack{q \leq 4 \pi \sqrt{N} \\ q \equiv 0 \Mod [d_1,d_2]}} \frac{1}{{\varphi(q)}^2} \sum_{j \leq \frac{\sqrt{N}}{q}} \frac{1}{j^2} \sum_{\substack{m_1, m_2 \leq 2M \\ m_3, m_4 \leq 2M \\ m_1 \neq m_2, \, m_1 \equiv m_2 \Mod d_1 \\ m_3 \neq m_4, \, m_3 \equiv m_4 \Mod d_2 \\ m_1 m_3 = m_2 m_4 }} \mathcal{S}^{(M)}_{\mathbf{m},j} \mathcal{T}^{(M)}_{\mathbf{m},j}. \label{equ:3.4} \tag{3.4}
\end{align*}
Again we use symmetry and the change of variables used in the proof of Lemma~\ref{l:novariation} (preceeding display~\eqref{equ:3.2}) to bound the innermost sum over $m_1, m_2, m_3$ and $m_4$ by
\begin{multline*} \sum_{a_2 \leq a_1 \leq 2M } \sum_{\substack{1 \leq a_3, a_4 \leq \frac{2M}{a_1}, a_3 \neq a_4 \\ a_3 \equiv a_4 \Mod \bigr[ \frac{d_1}{(a_1, d_1)}, \frac{d_2}{(a_2, d_2)} \bigl] }} \Biggl[  \min \Biggl\{\frac{j}{M^2}, \frac{j}{|\theta_1|} \Biggr\} \min \Biggl\{\frac{j}{M^2}, \frac{j}{|\theta_2|} \Biggr\} \\ \biggl( \frac{|\theta_1|^{1/2}}{M^{1/2}} + \frac{M^{3/2}}{|\theta_1|^{1/2}} \biggr) \biggl(\frac{|\theta_2|^{1/2}}{M^{1/2}} + \frac{M^{3/2}}{|\theta_2|^{1/2}} \biggr) \Biggr] ,
\end{multline*}
where $\theta_1 = a_1 j (a_3 - a_4)$ and $\theta_2 = a_2 j (a_4 - a_3)$.
\begin{remark}\label{r:rademachercase}
\normalfont~This is where we encounter difficulties in the case of Rademacher $f$: the change of variables gives ten parameters, as opposed to four in the Steinhaus case, and the summand is more complicated than in the proof of Lemma~\ref{l:novariation}.
\end{remark}
Noting that $|\theta_2| \leq |\theta_1|$, and performing a similar analysis as in the previous case, this is smaller than a positive constant times
\begin{align*}
\sum_{a_2 \leq a_1 \leq 2M} & \frac{M}{a_1} \sum_{1 \leq n \leq \frac{2M}{a_1 \bigl[\frac{d_1}{(a_1,d_1)}, \frac{d_2}{(a_2,d_2)}\bigr]}}\min \Biggl\{ \frac{j^2}{M^4}, \frac{j^2}{n \beta_1 M^2}, \frac{j^2}{n^2 \beta_1 \beta_2} \Biggr\} \biggl( \frac{n \beta_1^{1/2} \beta_2^{1/2}}{M} + \frac{\beta_1^{1/2} M}{\beta_2^{1/2}} + \frac{M^3}{n \beta_1^{1/2} \beta_2^{1/2}} \biggr), \label{equ:3.5} \tag{3.5}
\end{align*}
where $\beta_1 = a_1 j \bigl[\frac{d_1}{(a_1, d_1)}, \frac{d_2}{(a_2, d_2)}\bigr]$ and $\beta_2 = a_2 j \bigl[\frac{d_1}{(a_1, d_1)}, \frac{d_2}{(a_2, d_2)}\bigr]$. We split this sum into three parts according to the following cases and handle the contribution of each to~\eqref{equ:3.4}:
\begin{enumerate}[label = (\roman*)]
    \item $n \beta_1 \leq M^2$,
    \item $n \beta_1 \geq M^2$ and $n \beta_2 \leq M^2$,
    \item $n \beta_2 \geq M^2$. 
\end{enumerate}
In case (i) we take the first value in the minimum. The contribution of this case to~\eqref{equ:3.5} can then be bounded above by
\begin{align*} & \ll \sum_{a_2 \leq a_1 \leq 2M} \frac{j^2}{a_1 M^3} \sum_{n \leq \frac{M^2}{\beta_1}} \biggl( \frac{n \beta_1^{1/2} \beta_2^{1/2}}{M} + \frac{\beta_1^{1/2} M}{\beta_2^{1/2}} + \frac{M^3}{n \beta_1^{1/2} \beta_2^{1/2}} \biggr) \\
& \ll \sum_{a_2 \leq a_1 \leq 2M} \frac{j^2 \log N}{a_1 \beta_1^{1/2} \beta_2^{1/2}} = \sum_{a_2 \leq a_1 \leq 2M} \frac{j \log N}{a_1^{3/2} a_2^{1/2} \bigl[\frac{d_1}{(a_1, d_1)}, \frac{d_2}{(a_2, d_2)}\bigr]} \, . \end{align*}
The contribution of this to the larger sum~\eqref{equ:3.4} is
\begin{align*}
& \ll \frac{1}{{(\log N)}^{4 + 2 \varepsilon}} \sum_{d_1, d_2 \leq 2M} d_1 d_2 \sum_{\substack{q \leq 4 \pi \sqrt{N} \\ q \equiv 0 \Mod [d_1,d_2]}} \frac{1}{{\varphi(q)}^2} \sum_{a_2 \leq a_1 \leq 2M} \frac{1}{{a}_{1}^{3/2} {a}_{2}^{1/2} \bigl[\frac{d_1}{(a_1, d_1)}, \frac{d_2}{(a_2, d_2)}\bigr]} \sum_{j \leq \frac{\sqrt{N}}{q}} \frac{1}{j} \\
& \ll \frac{{(\log \log N)}^2}{{(\log N)}^{3 + 2 \varepsilon}} \sum_{d_1, d_2 \leq 2M} \frac{d_1 d_2}{{[d_1,d_2]}^2} \sum_{a_2 \leq a_1 \leq 2M} \frac{1}{a_1^{3/2} a_2^{1/2} \bigl[\frac{d_1}{(a_1, d_1)}, \frac{d_2}{(a_2, d_2)}\bigr]} \, .
\end{align*}
This bound is similar to~\eqref{equ:3.2}, and we again use the substitutions $g_i = (a_i, d_i)$, $D_i = d_i/g_i$ and $A_i = a_i/g_i$, dropping the coprime conditions as we did previously. We deduce that the previous display is
\[ \ll \frac{{(\log \log N)}^2}{{(\log N)}^{3 + 2\varepsilon}} \sum_{g_1 \leq 2M} \frac{1}{g_1^{5/2}} \sum_{A_1 \leq \frac{2M}{g_1}} \frac{1}{A_1^{3/2}} \sum_{g_2 \leq A_1 g_1} \frac{1}{g_2^{3/2}} \sum_{A_2 \leq \frac{A_1 g_1}{g_2}} \frac{1}{A_2^{1/2}} \sum_{\substack{D_1 \leq \frac{2M}{g_1} \\ D_2 \leq \frac{2M}{g_2}}} \frac{(D_1, D_2) {{(D_1 g_1, D_2 g_2)}^2}}{D_1^2 D_2^2}, \]
and summing over $A_2$ gives
\[ \ll \frac{{(\log \log N)}^2}{{(\log N)}^{3 + 2\varepsilon}} \sum_{g_1 \leq 2M} \frac{1}{g_1^{2}} \sum_{A_1 \leq \frac{2M}{g_1}} \frac{1}{A_1} \sum_{g_2 \leq A_1 g_1} \frac{1}{g_2^2} \sum_{\substack{D_1 \leq \frac{2M}{g_1} \\ D_2 \leq \frac{2M}{g_2}}} \frac{(D_1, D_2) {{(D_1 g_1, D_2 g_2)}^2}}{D_1^2 D_2^2} \ll \frac{{(\log \log N)}^2}{{(\log N)}^{1 + 2\varepsilon}}, \]
which follows from the proof of Lemma~\ref{l:novariation}, seeing as this term is the same as the sum in~\eqref{equ:3.3}. Therefore the contribution to~\eqref{equ:3.4} from case (i) is $o_{\varepsilon} \bigl( \frac{1}{\log N} \bigr)$. \\ Moving on, in case (ii) we take the second value in the minimum in~\eqref{equ:3.5}. The contribution from this case can be bounded above by
\begin{align*} & \ll \sum_{a_2 \leq a_1 \leq 2M} \frac{j^2}{a_1 \beta_1 M} \sum_{\frac{M^2}{\beta_1} \leq n \leq \frac{M^2}{\beta_2}} \biggl( \frac{\beta_1^{1/2} \beta_2^{1/2}}{M} + \frac{\beta_1^{1/2} M}{n \beta_2^{1/2}} + \frac{M^3}{n^2 \beta_1^{1/2} \beta_2^{1/2}} \biggr) \\
& \ll \sum_{a_2 \leq a_1 \leq 2M} \frac{j^2 \log N}{a_1 \beta_1^{1/2} \beta_2^{1/2}} = \sum_{a_2 \leq a_1 \leq 2M} \frac{j \log N}{a_1^{3/2} a_2^{1/2} \bigl[\frac{d_1}{(a_1, d_1)}, \frac{d_2}{(a_2, d_2)}\bigr]} \, , \end{align*}
This is identical to the contribution from case (i), so everything goes through similarly. \\ Finally, in case (iii), we take the last value in the minimum. The contribution to~\eqref{equ:3.5} is
\begin{align*} & \ll \sum_{a_2 \leq a_1 \leq 2M} \frac{j^2 M}{a_1 \beta_1 \beta_2} \sum_{\frac{M^2}{\beta_2} \leq n \leq \frac{2M}{a_1 \bigl[\frac{d_1}{(a_1,d_1)}, \frac{d_2}{(a_2,d_2)}\bigr]}} \biggl( \frac{\beta_1^{1/2} \beta_2^{1/2}}{n M} + \frac{\beta_1^{1/2} M}{n^2 \beta_2^{1/2}} + \frac{M^3}{n^3 \beta_1^{1/2}\beta_2^{1/2}} \biggr).
\end{align*}
Again, we find that this satisfies the same bound of
\[ \ll \sum_{a_2 \leq a_1 \leq 2M} \frac{j \log N}{a_1^{3/2} a_2^{1/2} \bigl[\frac{d_1}{(a_1, d_1)}, \frac{d_2}{(a_2, d_2)}\bigr]}, \]
so this last case is also handled similarly. Having show that~\eqref{equ:3.4} is $o_\varepsilon \bigl( \frac{1}{\log N} \bigr)$, the proof is completed.
\end{proof}
\subsection{Completion of the proof of Theorem~\ref{T:ub}}\label{s:completionoft2}
The theorem is a straightforward consequence of Proposition~\ref{p:mainprop}. Contrary to section~\ref{S:LB}, here we let $\tilde{\E}$ and $\tilde{\p}$ denote the expectation and probability conditioned on ${\bigl( f(p) \bigr)}_{p \leq N^{0.2}}$ respectively.
\begin{proof}[Proof of Theorem~\ref{T:ub}]
Let $\mathcal{A}$ be the event that
\[ \max_{\theta \in \mathcal{D}} \frac{1}{N} \sum_{N^{0.8} \leq p \leq N} \Bigl| \sum_{m \leq N/p} f(m) \e (mp\theta) \Bigr|^2 \leq C {{(\log N)}^{5/2 + \varepsilon}}, \]
where $C>0$ is the absolute constant from Proposition~\ref{p:mainprop}, so that for any $\varepsilon > 0$, we have ${\p} ( \mathcal{A}^c) = o_{\varepsilon} (1)$. As mentioned in section~\ref{sb:uboutline}, to prove Theorem~\ref{T:ub}, it suffices to prove the corresponding result where the maximum is taken only over $\mathcal{D}$, where $\mathcal{D}$ is as defined at the beginning of section~\ref{s:cvb}. Starting with
\[ \p \Bigl( \max_{\theta \in \mathcal{D}} \frac{1}{\sqrt{N}}\Bigl| \sum_{\substack{n \leq N \\ P(n) \geq N^{0.8}}} f(n) \e (n \theta) \Bigr| > {(\log N)}^{7/4 + \varepsilon} \Bigr), \]
we partition the probability space according to the event $\mathcal{A}$ and apply the union bound followed by a conditional Markov's inequality with $2k$-th moments (for $k \in \N$). This bounds the above quantity by
\begin{align*}
\leq \, & \p \Bigl( \max_{\theta \in \mathcal{D}} \frac{1}{\sqrt{N}}\Bigl| \sum_{\substack{n \leq N \\ P(n) \geq N^{0.8}}} f(n) \e (n \theta) \Bigr| > {(\log N)}^{7/4 + \varepsilon} \cap \mathcal{A} \Bigr) + \p (\mathcal{A}^c) \\
\leq \, & \sum_{\theta \in \mathcal{D}} \frac{1}{{(\log N)}^{(7/2 + 2\varepsilon) k}} \E \Bigl| \mathbf{1}_{\mathcal{A}} \frac{1}{\sqrt{N}} \sum_{\substack{n \leq N \\ P(n) \geq N^{0.8}}} f(n) \e (n \theta) \Bigr|^{2k} + o_{\varepsilon} (1) \\
\leq \, & \sum_{\theta \in \mathcal{D}} \frac{1}{{(\log N)}^{(7/2 + 2\varepsilon) k}} \E \Bigl[\mathbf{1}_{\mathcal{A}}\tilde{\E} {\Bigl| \frac{1}{\sqrt{N}} \sum_{N^{0.8} < p \leq N} f(p) \sum_{m \leq N/p} f(m) \e (mp \theta) \Bigr|}^{2k} \Bigr] + o_{\varepsilon} (1) .
\end{align*}
For each $k$-tuple of primes $(p_1,\ldots,p_k)$, there are $k!$ ways to choose $(p_{k+1},\ldots,p_{2k})$ so that these tuples have equal products. Therefore we have $\tilde{\E} {\bigl| \frac{1}{\sqrt{N}} \sum_{N^{0.8} < p \leq N} f(p)a_p \bigr|}^{2k} = k! {\bigl( \frac{1}{N} \sum_{N^{0.8} < p \leq N} |a_p|^{2} \bigr)}^k $. Seeing as $k! \leq k^k$, the previous display is
\begin{align*}
\leq \, & \sum_{\theta \in \mathcal{D}} \frac{k^k}{{(\log N)}^{(7/2 + 2\varepsilon) k}} \E \Biggl[\mathbf{1}_{\mathcal{A}} {\Bigl( \frac{1}{N} \sum_{N^{0.8} < p \leq N} \Bigl| \sum_{m \leq N/p} f(m) \e (mp \theta) \Bigr|^{2} \Bigr)}^k \Biggr] + o_{\varepsilon} (1) \\
\ll \, & \sum_{\theta \in \mathcal{D}} {\biggl( \frac{C k}{{(\log N)}^{1 + \varepsilon}}\biggr)}^k + o_{\varepsilon} (1).
\end{align*}
Taking $k = \lceil \log N \rceil$ and noting that $| \mathcal{D}| \ll N$, this is $o_{\varepsilon} (1)$, completing the proof of Theorem~\ref{T:ub}.
\end{proof}
\subsubsection*{Acknowledgements}
The author would like to thank his supervisor, Adam Harper, for many useful discussions and for carefully reading an earlier version of this work.
\subsubsection*{Rights Retention}
For the purpose of open access, the author has applied a Creative Commons Attribution (CC-BY) licence to any Author Accepted Manuscript version arising from this submission.
\printbibliography
\end{document}